\documentclass[12pt]{amsart}
\usepackage{lineno}
\usepackage{comment}

\usepackage{amssymb}
\usepackage{amsthm}
\usepackage{amssymb}

\newtheorem{thm}{Theorem}[section]

\newtheorem{prop}[thm]{Proposition}
\newtheorem{cor}[thm]{Corollary}
\newtheorem{lem}[thm]{Lemma}

\theoremstyle{remark}
\newtheorem*{rema}{Remark}

\theoremstyle{definition}\newtheorem{defi}[thm]{Definition}

\newcommand{\m}{(M,\omega)}
\newcommand{\mlam}{$(\widetilde M,\omega_\lambda)$}
\newcommand{\mtilde}{(\widetilde M,\widetilde\omega_\rho)}

\newcommand{\hf}{\textup{HF}}

\begin{document}


\title{Lagrangian Floer homology on  symplectic blow ups}
\author{Andr\'es Pedroza}
\address{Facultad de Ciencias\\
           Universidad de Colima\\
           Bernal D\'{\i}az del Castillo No. 340\\
           Colima, Col., Mexico 28045}
\email{andres\_pedroza@ucol.mx}

\begin{abstract}
We show how to compute the Lagrangian Floer homology in the 
one-point blow up of the proper transform of Lagrangians submanifolds,
solely in terms of information of the base manifold.  
As an example we present an alternative computation of the Lagrangian quantum homology
in the one-point blow up of $(\mathbb{C}P^2,\omega_{FS})$ of the proper transform of
the  Clifford torus.
\end{abstract}


\maketitle

\section{Introduction}
Lagrangian Floer homology (LFH for short) is a powerful tool developed by A. Floer 
\cite{floer-witten, floer-morse} to solve
the Arnol'd Conjecture about the minimal number of fixed points of a Hamiltonian 
diffeomorphism on a closed symplectic manifold.
Nevertheless, LFH is important in its own right
due to its rich algebraic structure and for its interference in the classification 
of Lagrangian submanifolds under Hamiltonian equivalence, among other things.
Unfortunately, it is extremely difficult to compute the LFH
of a  pair of Lagrangian submanifolds.

The aim of this article is to understand  how the LFH  changes
after symplectically blowing up a point that does not lie in the Lagrangian submanifolds. 
Further, we show how to compute the differential in the blown up manifold 
from the differential in the base manifold
and hence the LFH in 
the symplectic one-point blow up using {\em solely} data  from the base symplectic manifold.  
We used LFH for monotone Lagrangian submanifolds
as defined by Y.-G. Oh in \cite{oh-floercoho, oh-adde} and also the quantum
homology of a Lagrangian as defined by P. Biran and O. Cornea 
 in \cite{biran-cornea-rigidity} that was previously defined also by Y.-G. Oh in \cite{oh-relative}.
Hence, $(M,\omega)$ will stand for a symplectic manifold that is either closed
or convex at infinity and the Lagrangian submanifolds will be assumed to be monotone
 with  minimal Maslov number 
greater than or equal to two.

As for the Lagrangian submanifolds, beside the monotonicity condition we impose another 
condition which is directly related to the definition of the  symplectic one-point blow up. 
Recall that $L\subset (M,\omega)$ is called {\em monotone} if there exists a positive 
number $\lambda$ such 
that $\mu_L=\lambda\cdot \omega$ on $\textup{H}_2(M,L)$.
In order to remain in the context of monotone Lagrangians, the weight $\rho$ of the one-point blow up
must be equal to 
\begin{eqnarray}
\label{e:rho} 
\rho=\sqrt{\frac{n-1}{(\lambda/2)\cdot\pi}}.
\end{eqnarray}
Accordingly we focus on symplectic manifolds whose Gromov's width is such that
$c(M,\omega)>
2(n-1)/\lambda$.
Thus,  if 
$\iota:(B^{2n}(\rho),\omega_0)\to (M,\omega)$ is the symplectic embedding 
 used to define the symplectic one-point blow up $(\widetilde M,\widetilde\omega_\rho)$,
we restrict our analysis to monotone Lagrangian submanifolds $L\subset (M,\omega)$ such that
$L\cap \iota(B^{2n}(\rho))$ is empty. Here   $B^{2n}(\rho)\subset (\mathbb{R}^{2n},\omega_0)$ 
is the closed ball of radius $\rho$ centred at the origin.  
This means that the Lagrangian submanifold that we consider are such that 
$c(M\setminus L,\omega)>2(n-1)/\lambda$. 

This condition is quite restrictive. Here is an example that shows our limitation.
Consider $(\mathbb{C}P^n,\omega_\textup{FS})$
where the symplectic form is normalized so that the symplectic area of the line is $\pi.$
Then $\mathbb{R}P^n\subset (\mathbb{C}P^n,\omega_\textup{FS})$ 
 is a monotone Lagrangian submanifold 
with minimal Maslov number $n+1$. For $n>2$, the Lagrangian $\mathbb{R}P^n$ fails
the condition to be away from the prescribed embedded ball. 
For, according to 
P. Biran \cite[Theorem 1.B]{biran-lagrangianbarriers}, the image of any symplectic embedding
$(B^{2n}(\rho),\omega_0)\to (\mathbb{C}P^n,\omega_\textup{FS})$  with $\rho^2\geq 1/2$
intersects the Lagrangian $\mathbb{R}P^n$. On the other hand,  the value of the weight of 
the blow up, in order to obtain a monotone Lagrangian, must be equal to
 $\rho^2=(n-1)/(n+1)$. 
 Hence, for $n>2$ the monotone Lagrangian  
$\mathbb{R}P^n\subset (\mathbb{C}P^n,\omega_\textup{FS})$ 
fails the condition
to be disjoint from the embedded ball used to define the symplectic one-point blow up.

Let $\pi:(\widetilde M,\widetilde\omega_\rho)\to (M,\omega)$ denote  the blow up map. From now
on, we always assume that the weight of the blow up is such that $(\widetilde M,\widetilde\omega_\rho)$
is monotone and the Lagrangian submanifolds are such that the Gromov's width of their complement
is  greater than $2(n-1)/\lambda$.
Recall that away from the preimage of the embedded ball $\iota(B^{2n}(\rho))$, $\pi$ is a symplectic diffeomorphism.
Hence if $L\subset (M,\omega)$ is a 
Lagrangian submanifold such that  $L\cap  \iota(B^{2n}(\rho)) =\emptyset$,
then the proper transform $\pi^{-1}(L)$ is a Lagrangian submanifold in $(\widetilde M,\widetilde\omega_\rho)$.
Let $\widetilde L\subset (\widetilde M,\widetilde\omega_\rho)$
denote the Lagrangian submanifold $\pi^{-1}(L)$. If $L$ is monotone, 
it follows from  Lemma \ref{l:monotonelag}
that $\widetilde L$ is also monotone. However their minimal Maslov number do not have to 
coincide. 
In general the minimal Maslov number of $\widetilde L$
decreases in comparison with the minimal Maslov number of  $L$,
 $N_{\widetilde L}\leq N_L.$

For instance consider $L=\mathbb{R}P^2$ in $(\mathbb{C}P^2,\omega_\textup{FS})$.
Thus 
$L$ is monotone and  its minimal Maslov number is three. 
Further, there is a symplectic embedding
of $(B^4(1/\sqrt{3}),\omega_0)$ into  $(\mathbb{C}P^2,\omega_\textup{FS})$ 
whose image does not intersect $L$.
 Let $A\in \textup{H}_2(\mathbb{C}P^2,\mathbb{R}P^2)$ be such that
$\mu_L(A)=3$ and $\widetilde A\in \textup{H}_2(\widetilde{\mathbb{C}P}\,^2,\widetilde L)$ the proper
transform of $A$. By the definition of the Maslov index if follows that 
$\mu_{\widetilde L}(\widetilde A)=3$.
If $L_E\in \textup{H}_2(\widetilde{\mathbb{C}P}\,^2)$ is the class of the exceptional line,
then $c_1(\widetilde{\mathbb{C}P}\,^2)(L_E)=-1$ and
\begin{eqnarray*}
\mu_{\widetilde L}(\widetilde A\#L_E )=
\mu_{\widetilde L}(\widetilde A) +2c_1(\widetilde{\mathbb{C}P}\,^2)(L_E)=1.
\end{eqnarray*}
That is $N_{\widetilde{\mathbb{R}P}\,^2 }=1$. 
In particular
its LFH is not well-defined in this case.

\begin{defi}
Let $(M,\omega)$ be a symplectic manifold of dimension $2n\geq 4$ that is either closed
or symplectically convex at infinity. A collection of closed Lagrangian submanifolds
 $L_0,\ldots, L_k \subset (M,\omega)$
is called {\em admissible} in $(M,\omega)$ if 
\begin{itemize}
\item[a)] the Lagrangians are monotone 
 with the same monotonicity constant $\lambda$, 
\item[b)] $c(M\setminus (L_0\cup\cdots \cup L_k),\omega)>2(n-1)/\lambda$ and
\item[c)] $N_{\widetilde L_j}\geq 2$ for $j\in\{1,2\ldots ,k\}$ where the one-point
blow up of $(M,\omega)$ is taken with respect to any symplectic embedding of the ball
in $M\setminus (L_0\cup\cdots \cup L_k)$ whose radius is given by (\ref{e:rho}).
\end{itemize}
\end{defi}
 
Consider $L_0,L_1\subset (M,\omega)$  Lagrangians submanifolds  that are
admissible.
The goal of this note  is to compute $\textup{HF}(\widetilde{L}_0,\widetilde{L}_1)$
in terms of the based manifold $(M,\omega)$, the Lagrangians $L_0$ and $L_1$, and the Floer
complex $(\textup{CF}(L_0,L_1),\partial)$. (As mentioned above, the pearl complex
will also be considered.) 
The first step is to understand  the role played by the 
blown up point. The blown up point not only has to lie in the complement of
the Lagrangian submanifolds, it must be a generic point.
The reason been that  any pair of points in 
the complement of $L_0\cup L_1\subset
(M,\omega)$ 
can be mapped one to the other by a Hamiltonian diffeomorphism that  leaves the Lagrangian submanifolds fixed.

\begin{thm}
\label{t:invariance}
Let $L_0,L_1\subset (M,\omega)$  be admissible  Lagrangian submanifolds.
Assume that there exists 
$J\in \mathcal{J}_\textup{reg}(L_0,L_1)$ 
such that  $\iota^*J=J_0$ where $\iota:(B^{2n}(\rho),\omega_0)\to (M,\omega)$ is 
a symplectic embedding that avoids the
Lagrangian submanifolds and $\rho$ is given by Eq. (\ref{e:rho}).
If
 $\pi_j: (\widetilde M_j,\widetilde\omega_j)\to (M,\omega)$
are  the  monotone one-point blow up of
$(M,\omega)$ at $x_j\in M\setminus (L_0\cup L_1)$ for $j=1,2$,
then
\begin{eqnarray*}
\textup{HF}_*(\pi_1^{-1}(L_0),\pi_1^{-1}(L_1))\simeq
\textup{HF}_*(\pi_2^{-1}(L_0),\pi_2^{-1}(L_1))
\end{eqnarray*}
as $\Lambda$-modules.
\end{thm}

\begin{rema}
The hypothesis on the almost complex structure, $\iota^*J=J_0$,  is required in order
to  move holomorphic information from the based manifold to the blow up manifold and vice versa.
Here, $J_0$ stands for the standard complex
structure on $\mathbb{C}^n.$ 
Actually, is not hard to obtain an almost complex
structure $J$ such that satisfies  $\iota^*J=J_0$.
 The problem is to guarantee
that such almost complex structure is regular. Although, $\mathcal{J}_\textup{reg}(L_0,L_1)$ 
is a dense set in $\mathcal{J}(M,\omega)$
the set $\{J\in \mathcal{J}(M,\omega) | \iota^*J=J_0\}$ is closed and has a
dense complement. 
On the other side, most known examples of symplectic manifolds that at the same time
are  complex manifolds, the complex structure happens to be regular.
\end{rema}

Hence  $\textup{HF}_*(\widetilde{L}_0,\widetilde{L}_1)$ is independent
of the blown up point of $(M,\omega)$ in the 
complement of $L_0\cup L_1$. 
Now that the lack of 	relevance of the blown up point has been settled, we focus con
the comparison of the LFH on the based manifold and the one-point blow up.
By the observation that the blow up map is a symplectic diffeomorphism 
away from the embedded ball, heuristically
$\textup{HF}_*(L_0,L_1)$ is isomorphic to $\textup{HF}_*(\widetilde L_0,\widetilde L_1)$
if the blown up point is not implicated in the  data involved in the definition of 
$\textup{HF}_*(L_0,L_1)$.
 The precise statement is the following.

\begin{thm}
\label{t:noholo}
Let $L_0,L_1\subset (M,\omega)$  be admissible Lagrangians.
If there exists $x_0\in M\setminus (L_0\cup L_1)$ and
 $J\in\mathcal{J}_{\textup{reg}}(L_0,L_1)$
such that $\iota^*J=J_0$; where $\iota$ is the symplectic embedding of the ball,
 and $x_0$ does not lie
in any $J$-holomorphic strip of Maslov-Viterbo index $2n-1$, then
\begin{eqnarray*}
\textup{HF}_*(\widetilde L_0,\widetilde L_1)\simeq
\textup{HF}_*(L_0,L_1).
\end{eqnarray*}
as $\Lambda $-modules.
\end{thm}

\begin{rema}
\begin{itemize}
\item 
A straightforward class of examplea of the above result is the case when the
symplectic manifold is non compact. For instance $(\mathbb{C}^n,\omega_0)$
and any pair of compact Lagrangian submanifolds that meet the hypothesis of 
Theorem \ref{t:noholo}. That is, after blowing up one point in $(\mathbb{C}^n,\omega_0)$,
the induced compact Lagrangian submanifolds  in  $(\widetilde{\mathbb{C}}^n,\widetilde \omega_0)$
will continue to be displaceable.

\item 
The reason of the value of $2n-1$ of the Maslov-Viterbo index of the holomorphic strip
is the following. The image of $\mathbb{R}\times [0,1]\sqrt{-1}$ under the set of all
$J$-holomorphic strips in the same class of Maslov-Viterbo index $2n-1$,
{\em generically} is $2n$-dimensional and after blowing up it could induced
a holomorphic strip on the one-point blow up of Maslov-Viterbo index 1.
This claim is proved in Proposition \ref{p:maslovindx}.
Basically, this is the only possibility to alter the differential of the Floer complex of
the blow up in comparison with  the differential of the based manifold.
\item By Theorem \ref{t:invariance}, the LFH is independent of the blown up
point, thus some uniformity on the complement of $L_0\cup L_1$
is required in order to alter the Floer differential. 
This will be done below by assuming a uniruled condition
in
 the context of quantum homology.
\end{itemize}
\end{rema}

From Theorem \ref{t:invariance}, $\textup{HF}_*(\widetilde L_0,\widetilde L_1)$
is independent
of the point in $(M,\omega)$ that is blown up. Further, from Theorem \ref{t:noholo}
holomorphic strips of Maslov-Viterbo index $2n-1$ are the strips that
can affect the differential in the blown up manifold in comparison with the differential
in the base manifold.
Thus a necessary condition for  $\textup{HF}_*(\widetilde L_0,\widetilde L_1)$
to be non isomorphic to  $\textup{HF}_*(L_0,L_1)$ is the following: 
for any $x_0\in M\setminus (L_0\cup L_1)$
there exists a dense set $\mathcal{J}_\textup{reg}(L_0,L_1)$ of regular almost 
complex structures such that for every $J\in \mathcal{J}_\textup{reg}(L_0,L_1)$
there exists a $J$-holomorphic strip $u:\mathbb{R}\times [0,1]\sqrt{-1}\to M$ of 
Maslov-Viterbo index $2n-1$
such that $x_0\in u(D)$ with the usual boundary conditions. In this paper we do not
explore this condition in general. Instead we consider the case $L_0=L_1$.
Thus from now on,  we only consider the case $\textup{HF}_*(L,L)$.
Further, in order to avoid the Hamiltonian perturbation that is 
involved in the calculation of
$\textup{HF}_*(L,L)$ and present a clear picture of the phenomenon implicated
in our  computation we used  quantum homology $\textup{QH}_*(L)$ instead
 of $\textup{HF}_*(L,L)$. Thus $\Lambda$ stands for two different rings depending
it $\textup{HF}_*(L,L)$ or $\textup{QH}_*(L)$ is being used.

From Theorem \ref{t:invariance}, $\textup{QH}_*(\widetilde L)$ is independent
of the point in $(M,\omega)$ that is blown up. Further, the condition that $x_0$ does not
lie in any $J$-holomorphic strip of Maslov-Viterbo index $2n-1$ of Theorem \ref{t:noholo}
is replace by the condition that  $x_0$ does not lie in 
any $J$-holomorphic disk  of Maslov index $2n.$ In this case, $\textup{QH}_*(\widetilde L)
\simeq \textup{QH}_*(L)$. Henceforth, pearly trajectories that contain a 
$J$-holomorphic disk of Maslov index $2n$ induced new pearly trajectories
in the one-point blow up.
That is, 
the only possibility
to for $\textup{QH}_*(\widetilde L)$   to be distinct from 
$\textup{QH}_*( L)$, or at least to have new pearly trajectories
on the blow up
is if all points in  complement of $L\subset (M,\omega)$ are
look-alike from the Floer homology perspective.

This naturally leads to the concept of
uniruled introduced by P. Biran and O. Cornea \cite[Definition 1.1.2]{biran-cornea-rigidity}. 
The pair $(M,L)$ is said to be {\em $(p,q)$-uniruled  of order $k$} 
if   for any  $p$ distinct points $p_j\in M\setminus L$ and 
any $q$ distinct points $q_j\in L$, 
 there exists a subset
$\mathcal{J}_{\textup{reg}}\subset \mathcal{J}(M,\omega)$ of second category
such that for each $J\in \mathcal{J}_{reg}$ there exists
a $J$-holomorphic map $u:(D,\partial D)\to (M,L)$ such that
 $p_j\in u(D)$, $q_j\in u(\partial D)$ and $\mu_L([u])\leq k$.
Although the uniruled condition is very restrictive, it is the class of examples where
the LFH on the the blow up can differ from the LFH on the base manifold.

The next result summarise how the pearly differential 
changes in the pearl complex  blow up manifold compared with the
differential of the pearl complex in the base manifold.
For simplicity of the statement we state the result
for four-manifolds.

\begin{thm}
\label{t:dif}
Let $L$ be an admissible Lagrangian in $(M,\omega)$ that is
$(1,2)$-uniruled of order $2$, $x_0\in M\setminus L$ a generic point and suppose that $\textup{dim}(M)=4$. 
If $p$ and $q$ in $L$ are critical points of a generic Morse-Smale function $f$
with respect to a generic Riemannian metric $g$ on $L$ such that $\textup{ind}(p)-\textup{ind}(q)-1=-2$, then
$$
\langle \widetilde d(p), q\rangle=  \langle  d(p), q\rangle +_{\mathbb{Z}_2} k
$$
where $k$ is the number, mod $2$, of classes $A\in \textup{H}_2(M,L;\mathbb{Z})$ such that for some 
$J\in \mathcal{J}_\textup{reg}(M,L)$,
 such that $\iota^*J=J_0$ where $\iota$ is the symplectic embedding of the ball,
the moduli space of pearly trajectories
$ \mathcal{P}(p,q,A;g,f,J)$ is non empty,
 $\mu_L(A)=4$ and there is a $J$-holomorphic disk $u$ such that
 $x_0\in u(D) $ and $A=[u]$.
 \end{thm}


Our technique gives an alternative proof  of the fact that the Lagrangian
induced by the Clifford  torus
in the one-point blow up of $(\mathbb{C}P^2,\omega_\textup{FS})$ is wide.
Since the induced Lagrangian torus on the  
one-point blow of $(\mathbb{C}P^n,\omega_\textup{FS})$ is a toric fiber, this
fact follows from the theory of K. Fukaya, Y.-G. Oh, H. Ohta and K. Ono of 
\cite{FO3-toricI}. It was also proved by M. Entov and L. Polterovich in
 \cite{entov-pol-rigid}.
Recall from \cite[Cor. 1.2.12]{biran-cornea-rigidity} that the Clifford torus  
$\mathbb{T}_\textup{Cliff}\subset
(\mathbb{C}P^2,\omega_\textup{FS})$ is $(1,1)$-uniruled of order 4.
The prof of the next result is given in Section \ref{s:clif}.

\begin{thm}
\label{t:cliff}
After blowing up one point the proper transform of the Clifford torus $\mathbb{T}_\textup{Cliff}\subset
(\mathbb{C}P^2,\omega_\textup{FS})$
in $(\widetilde{\mathbb{C}P}\,^2,\widetilde\omega_\rho)$ is also a wide Lagrangian.
\end{thm}

It is also interesting to study the same problem with the possibility of varying the weight of the blow up.
However in order to do this, it will be required to use the theory of 
Kuranishi structures of
 K. Fukaya, Y.-G. Oh, H. Ohta and K. Ono \cite{fooo, foooII}
for LFH and to consider unobstructed Lagrangian submanifolds.
No attempt to understand such phenomenon is presented here.

\section{Review of $\textup{HF}_*(L_0,L_1)$ and $\textup{QH}_*(L)$}
\label{s:lfh }
\subsection{Lagrangian Floer Homology}
Throughout this note  $(M,\omega)$  will denote a symplectic manifold that is
  either closed
or convex at infinity,
$J=\{J_t\}_{0\leq t\leq 1}$
a family of $\omega$-compatible almost complex structures, and
$L_0$ and $L_1$
compact  connected Lagrangian submanifolds that  intersect  transversally.
Let $\mathcal{X}(L_0,L_1)$ denote the set of  intersection
points. Then for $p$ and $q$ in  $\mathcal{X}(L_0,L_1)$
and  $\beta\in \pi_2(M, L_0\cup L_1)$
denote by $\widehat{\mathcal{M}}(p,q,  \beta, J)$ the set of smooth maps
$u:\mathbb{R}\times [0,1]\to M $ such that:

\begin{itemize}
\item satisfy the boundary conditions
$$
u(s,0)\in L_0, \textup{ and } u(s,1)\in L_1,   \textup{ for all }s\in \mathbb{R}
$$
and
$$
\lim_{s\to -\infty} u(s,t) =q \textup{ and }
\lim_{s\to +\infty} u(s,t) = p;
$$
\item represent the class $\beta$,  $[u]=\beta$ and
\item are $J$-holomorphic,
$$
\overline\partial_J(u):=\frac{•\partial }{\partial s}u(s,t)
+ J_t \frac{•\partial }{\partial t}u(s,t) =0.
$$
\end{itemize}

The moduli space $\widehat{\mathcal{M}}(p,q,  \beta, J)$ admits an action
of $\mathbb{R}$, given by $r.u(s,t)=u(s-r,t)$.
Denote by ${\mathcal{M}}(p,q,  \beta, J)$
the quotient space of the action.
Elements of ${\mathcal{M}}(p,q,  \beta, J)$ are called holomorphic
strips; they also received the name of holomorphic
disks
since $\mathbb{R}\times [0,1]i$ is conformally equivalent to the
closed disk minus two points on the boundary.

In some cases the space $\widehat{\mathcal{M}}(p,q,  \beta, J)$ is in fact a smooth manifold.
To that end, take into account the linearised operator $D_{\overline\partial,u}$ of $\overline\partial_J$ at
$u\in \widehat{\mathcal{M}}(p,q,  \beta, J)$. Then for integers $k$ and $p$ such that $p>2$ and
 $k>p/2$
we have  the Sobolev space of vector  fields whose $k$-weak derivatives
exist and lie in $L_p$, and with boundary restrictions;
$$
W_k^p(u^*TM;L_0,L_1)\!:=\!\{ \xi\!\in W_k^p(u^*TM) \left|  \xi(s,0)\!\in TL_0,  \xi(s,1)\in TL_1
\textup{ for all } s\in\mathbb{R}
\right.\! \}.
$$
There exists $\mathcal{J}_\textup{reg}(L_0,L_1)$,  a dense subset of $\omega$-compatible almost complex
structures  
 in $ \mathcal{J}(L_0,L_1)$ such that for
$J\in \mathcal{J}_\textup{reg}(L_0,L_1)$
the linearised operator
$$
D_{\overline\partial(J),u}: W_k^p(u^*TM;L_0,L_1) \to L_p(u^*TM).  
$$
is Fredholm and surjective for all $u\in\widehat{\mathcal{M}}(p,q,  \beta, J)$.
Elements of $\mathcal{J}_\textup{reg}(L_0,L_1)$ are called regular.
In this case the index of $D_{\overline\partial(J),u}$ equals the Maslov index $\mu_{L_0,L_1}([u])$
of the homotopy type of $u$ in $\pi_2(M,L_0\cup L_1)$. Note that in the case when $J$ is regular
the dimension of the moduli space $\widehat{\mathcal{M}}(p,q,  \beta, J)$ is independent
of the regular  almost complex structure.

Let
$\textup{CF}(L_0,L_1)$ denote  the $\Lambda$-vector space generated by the intersection points
$\mathcal{X}(L_0,L_1)$.
Here $\Lambda$ stands for  the Novikov field
$$
\Lambda:=
\left
\{
  \sum_{j=0}^\infty a_j T^\lambda_j \left| a_j\in \mathbb{Z}_2,
\lambda_j\in \mathbb{R}, \lim_{j\to \infty}\lambda_j=\infty\right.  \right\}. 
$$
In the case when $[u]\in {\mathcal{M}}(p,q,  \beta, J)$,
 $\mu_{L_0,L_1}([u])=1$ and $J$ is regular
the moduli space ${\mathcal{M}}(p,q,  \beta, J)$ is 0-dimensional and compact,
 thereby  a finite
set of points.
Set $\#_{\mathbb{Z}_2} {\mathcal{M}}(p,q,  \beta, J) $ to
be the  module $2$ number of points of ${\mathcal{M}}(p,q,  \beta, J) $. The
Floer differential $\partial_J: \textup{CF}(L_0,L_1)\to \textup{CF}(L_0,L_1)$
is defined as
\begin{eqnarray}
\label{e:partial}
\partial_J (p): =\sum_{
\substack{
            q\in \mathcal{X}(L_0\cap L_1)\\
            [u] :     
            \textup{index}([u])=1
            }
}\#_{\mathbb{Z}_2} {\mathcal{M}}
(p,q,  [u], J) \,T^{\omega([u])} \,q.
\end{eqnarray}

If the Lagrangian submanifolds
$L_0$ and $L_1$ are monotone
and the minimal Maslov number of $L_0$ and $L_1$  
is greater than or equal to three, then $\partial_J\circ\partial_J=0$.
In this case the Lagrangian  Floer homology of $(L_0,L_1)$ is defined as
$$
\hf(L_0,L_1):=\frac{\textup{ker }\partial_J}{\textup{im }\partial_J}.
$$
Is important to note that the homology group $\textup{HF}(L_0,L_1)$ does not depend on the
regular $\omega$-compatible almost complex structure. 
When $L_0$ and $L_1$ are Hamiltonian
isotopic and the minimal Maslov number is two, the differential also squares to zero and 
 the Lagrangian  Floer homology is also defined in this case.

The role played by the coefficient field $\Lambda$ becomes essential in the definition of
the differential $\partial$. In principle the sum in Eq. (\ref{e:partial})
can be infinite, but by Gromov's compactness there are only finitely many homotopy classes
whose energy is below a determined value.  Hence $\Lambda$ assures that
Eq. (\ref{e:partial}) is well-defined.

For further details in the definition  of Lagrangian Floer homology in the monotone case
see \cite{oh-floercoho}; and also \cite{auroux-a} and \cite{fooo} for a more broader class of symplectic manifold
where Lagrangian Floer homology is defined.

\subsection{Lagrangian quantum homology}
The Lagrangian quantum homology is an invariant of a {\em single} Lagrangian submanifold.
Consider $L\subset (M,\omega)$ a closed monotone Lagrangian submanifold of minimal Maslov number
greater than  or equal to two. 
Further, fix a Riemannian metric $g$ on $L$ and Morse-Smale function
$f:L\to\mathbb{R}$. And on $(M,\omega)$ fix  an $\omega$-compatible almost complex
structure $J\in\mathcal{J}(M,\omega)$.

In the setting of quantum homology of $L\subset (M,\omega)$, $\Lambda$ stands
for $\mathbb{Z}_2[t,t^{-1}]$ where the degree of $t$ is $-N_L.$ Since from the 
context it will be clear if we are considering LFH or quantum homology; we make no
distinction on the notation of $\Lambda$.

Consider the $\Lambda$-module generated by the critical points of $f$
\begin{eqnarray*}
\mathcal{C}(g,f,J):=\mathbb{Z}_2\langle \textup{Crit}(f)\rangle\otimes\Lambda
\end{eqnarray*}
So far, only information about $L$ has been used. 
The ambient manifold $(M,\omega)$ 
enters in the definition of the differential. Let $\textup{H}_2^D(M,L)$ denote the image of
the Hurewicz homomorphism $\pi_2(M,L)\to \textup{H}_2(M,L)$. 
Fix $x,y\in \textup{Crit}(f)$ and $A\in \textup{H}_2^D(M,L)$ with $A\neq 0$. An {\em $A$-pearly trajectory} from $x$ to $y$,
denoted by $(u_1,\ldots,u_r; l_0,\ldots l_{r})_A$, is by definition a collection of
\begin{itemize}
\item $u_j:(D,\partial D)\to (M,L)$ non constant $J$-holomorphic disks for $j\in\{ 0,\ldots, r\}$,
\item $l_0:(-\infty,b_0]\to L$ and  $l_r:[a_r,\infty)\to L$  flow rays of $-\nabla f$, 
\item $l_j:[a_j,b_j]\to L$ flow cords of $-\nabla f$ for $j\in\{ 1,\ldots, r-1\}$.
\end{itemize}
These objects are related by the following conditions:
\begin{itemize}
\item $A=[u_1]+\cdots +[u_r]$,
\item
$x=\lim_{t\to-\infty}l_0(t)$, 
\item
$y=\lim_{t\to\infty}l_r(t)$, 
\item 
$l_j(b_j)=u_{j-1}(-1)$ for $j\in\{ 0,\ldots, r-1\}$, and
\item 
$l_j(a_j)=u_{j}(1)$ for $j\in\{ 1,\ldots, r\}$.
\end{itemize}

Two $A$-pearly trajectories 
$(u_1,\ldots,u_r; l_0,\ldots l_{r})_A$ and $(u^\prime_1,\ldots,u^\prime_s; l_0^\prime,\ldots l_{s}^\prime)_A$
are said to be equivalent if $r=s$ and for every $j\in\{1,\ldots, r\}$ there exists $\sigma_j\in \textup{Aut}(D)$
such that $\sigma_j(-1)=-1, \sigma_j(1)=1$ and $u^\prime_j=u_j\circ\sigma_j.$
The collection of all equivalence classes of $A$-pearly trajectories  is denoted by $\mathcal{P}(x,y,A;g,f,J)$.
In the case when $A=0$, by definition, the space $\mathcal{P}(x,y,0;g,f,J)$ consists of
unparameterised  flow lines of $-\nabla f$ that connect $x$ to $y$. 
That is, Morse trayectories.

The space  $\mathcal{P}(x,y,A;g,f,J)$ is a smooth manifold of dimension
\begin{eqnarray*}
\delta(x,y,A):= \textup{ind}_f(x)-\textup{ind}_f(y)-1 +\mu(A). 
\end{eqnarray*}
where $\textup{ind}_f(\cdot)$ stands for the Morse index with respect to $f$ of the critical point.
Recall that the Riemannian metric $g$ and $f:L\to\mathbb{R}$   are subject to the condition
that $f$ is a Morse-Samale function. If follows from \cite[Prop. 3.2]{biran-octav-lagquan},
see also \cite{biran-cornea-rigidity}, that there is dense subset $\mathcal{J}_\textup{reg}(M,L)$ in the
space of $\omega$-compatible almost complex structures $\mathcal{J}(M,\omega)$ such that for $J\in 
\mathcal{J}_\textup{reg}(M,L)$ if $\delta(x,y,A)=0$ then $\mathcal{P}(x,y,A;g,f,J)$ is a finite
set.

Fix $J\in \mathcal{J}_\textup{reg}(M,L)$. Then for $x\in \textup{Crit}(f)$ the differential
is defined as
\begin{eqnarray*}
d_J(x):= \sum_{y\in \textup{Crit}(f),\ \ A\in \textup{H}^D_2(M,L)} \#_{\mathbb{Z}_2} \mathcal{P}(x,y,A;g,f,J) 
\, T^{\mu(A)/N_L} \, y
\end{eqnarray*}
where the sum is over all $y$ and $A$ such that $\delta(x,y,a)=0$. Further, $d_J$ extends
$\Lambda$-linearly to $\mathcal{C}(g,f,J)$.  
 $(\mathcal{C}(g,f,J), d_J)$ is called {\em the 
 pearl complex of $(M,L,g,f,J)$}. It's homology is called the {\em Lagrangian quantum homology}
 of $L\subset (M,\omega)$
 \begin{eqnarray*}
\textup{QH}_*(L):=
\textup{H}_*(\mathcal{C}(g,f,J), d_J)
=
\frac{\textup{ker} d_J}{\textup{im} d_J}.
\end{eqnarray*}

Further details about the definition of 
the  $\Lambda$-module $\textup{QH}_*(L)$,  as well as the invariance 
of the Riemannian metric $g$, the Morse-Smale function $f$ and 
$J\in \mathcal{J}_\textup{reg}(M,L)$,  can be found in  
\cite{biran-octav-lagquan,biran-cornea-rigidity}.


\section{Review of the symplectic one-point blow up}
\label{s:symplecticblow}

The symplectic one-point blow up plays a fundamental role
 in this note. Hence we review
the definitions of the  complex  and  symplectic one-point blow up.
To that end,
consider the complex blow up of $\mathbb{C}^n$ at the origin
 $\Phi: \widetilde{\mathbb{C}^n}\to \mathbb{C}^n$, where $n>1$.
 That is
$$
\widetilde{\mathbb{C}^n}= \{(z,\ell)\in   \mathbb{C}^n \times \mathbb{C}P^{n-1} \; | \; z\in \ell
\}
$$
and the blow up map is given by $\Phi(z,\ell)=z$.  For $r>0$, let
$L(r):=\Phi^{-1}(\textup{int}(B^{2n}(r)))$ where  $B^{2n}(r)\subset \mathbb{C}^n$ is the closed ball and
 $\textup{int}  (\cdot)$ stands for the interior of the
set.

If $(M,J)$ is a complex manifold   and $\iota: (\textup{int}B^{2n}(r),J_0)\to
(M,J)$ is such that $\iota^*J=J_0$ and $x_0=\iota (0)$, then the complex blow up of $M$ at
$x_0$ is defined as
\begin{eqnarray}
\label{e:blow}
\widetilde M:= \left( M\setminus \{x_0\} \right)
\cup L(r)/\sim
\end{eqnarray}
where $z=\iota(z^\prime)\in
\iota (\textup{int}(B^{2n}(r))) \setminus \{x_0\}$ is identified with the unique point $(z^\prime,\ell_{z^\prime})\in L(r)$
and $\ell_{z^\prime}$ is the line determined by $z^\prime$.
So defined, $\widetilde M$ carries a unique complex structure
$\widetilde J$
 such that the blow up map
$\pi: (\widetilde M,\widetilde J)\to (M,J)$ is $(\widetilde J,J)$-holomorphic.
The preimage of the blown up point
 $\pi^{-1}(x_0)=E$ is called
the exceptional divisor. Further, the blow up map induces
 a  biholomorphic map $\widetilde M\setminus E
\to M\setminus\{x_0\}$.

The next task is to define the symplectic one-point blow up.  The symplectic blow up
relies on the complex blow up. However  there is not a unique symplectic blow up;
there is a whole family of symplectic forms on the one-point blow up.

As a first step we look at $(\mathbb{C}^{n},\omega_0)$ and
define a symplectic structure on  $\widetilde{\mathbb{C}^n}$.
Here
$\omega_0$ is the standard symplectic form on   euclidean space.
For $\rho>0$, consider  the symplectic form
$$
\omega(\rho):=\Phi^*(\omega_0)+\rho^2 pr^*(\omega_\textup{FS})
$$
 on $\widetilde{\mathbb{C}^n}$ where
$pr:\widetilde{\mathbb{C}^n}\to \mathbb{C}P^{n-1}$ is the canonical line bundle
and the Fubini-Study form $(\mathbb{C}P^{n-1},\omega_\textup{FS})$ is normalised so that
the area of every line is $\pi$.
Note that on the exceptional divisor the symplectic form $\omega(\rho)$ restricts to
$\rho^2 \omega_\textup{FS}.$
Hence in $(\widetilde{\mathbb{C}^n}, \omega(\rho)) $ the area of any
 line in the exceptional divisor is $\rho^2 \pi$.

Next, the symplectic form $\omega(\rho)$ is perturb in such a way so that
in the complement of a neighbourhood of the exceptional divisor agrees
with the standard symplectic form $\omega_0$.
Once this is done,
following the definition of the blow up manifold (\ref{e:blow}) it will
be possible to define a symplectic form on $\widetilde M$.

For $r>\rho$
let $\beta:[0,r]\to [\rho, r]$ be any
smooth function such that
$$
\beta(s) :=
\left\{
	\begin{array}{ll}
		\sqrt{\rho^2+s^2} & \mbox{for } 0\leq s\leq \delta \\
		s & \mbox{for  } r-\delta \leq s\leq r.
	\end{array}
\right.
$$
and on the remaining part takes any value as long as  $0<\beta^\prime(s)\leq 1$
for $0<s\leq r-\delta$.
Then  $F_\rho: L(r)\setminus E\to  \textup{int}(B^{2n}(r))\setminus B^{2n}(\rho)$   defined as
$$
F_\rho(z):= \beta(|z|)\frac{z}{|z|}
$$
is a diffeomorphism  such that
$\widetilde\omega(\rho):=F_\rho^*(\omega_0)$ is a symplectic form. So defined $\widetilde\omega(\rho)$
is  such that
\begin{itemize}
\item $\widetilde\omega(\rho)=\omega_0$
 on $\textup{int} (L(r) \setminus L(r-\delta) )$ and
\item $\widetilde\omega(\rho)=\omega(\rho)$ on $L(\delta)$.
\end{itemize}
We call $(  L(r), \widetilde\omega(\rho))$
the local model of the symplectic one-point blow up.

In order to define the symplectic blow up of $(M,\omega)$ at $x_0$,  it is requiere
a symplectic embedding  $\iota: (B^{2n}(\rho),\omega_0)\to (M,\omega)$
and an almost complex structure $J$ on $(M,\omega)$ such that
$\iota(0)=x_0$
and $\iota^*J=J_0$.
Notice that the symplectic embedding  $\iota$ extends to $\textup{int}  (B^{2n}(r))$ for $r$ such that
$r-\rho$ is small.


Finally, using the symplectic embedding as a symplectic chart and the local model
$(  L(r), \widetilde\omega(\rho))$ defined above,
 the symplectic form
of weight
$\rho$ on $\widetilde M$ is defined as
$$
\widetilde \omega_\rho :=
\left\{
	\begin{array}{ll}
		\omega  & \mbox{on } \pi^{-1}(M \setminus \iota B^{2n}(\sqrt{\rho^2+\delta^2}))\\
		\widetilde\omega(\rho) & \mbox{on  } L_{r}.
	\end{array}
\right.
$$
For further details and the dependence of the symplectic blow up on the choices that we made see
\cite{mcduffpol-packing}, \cite{ms} and \cite{msjholo}. The above observations are summarised
in the next proposition.

\begin{prop}
\label{p:propblow}
Let $(M,\omega)$ be a symplectic manifold,
$\iota: (B^{2n}(r), \omega_0)\to (M,\omega)$  a symplectic embedding  and $J$
a $\omega$-compatible almost complex structure such that
 $\iota(0)=x_0$ and $\iota^*J=J_0$. If
$\rho<r$, then  the symplectic blow up $\pi:(\widetilde M,\widetilde\omega_\rho)\to (M,\omega)$  of weight $\rho$
satisfies:
\begin{enumerate}
 \item $\pi: \widetilde M\setminus E\to M\setminus \{x_0\}$ is a diffeomorphism,
 \item $\pi^{*}(\omega)=\widetilde \omega_\rho$ on $\pi^{-1}(M\setminus \iota B^{2n}(r))$, and
 \item the area of the line in $E$ is $\rho^2\pi$.
\item $\widetilde J$ is  $\widetilde \omega_\rho$-compatible.
\end{enumerate}
\end{prop}

From now on assume that $J$ on $(M,\omega)$ satisfies  the condition $\iota^*J=J_0$,
where $J_0$ is the standard complex structure on $\mathbb{C}^n$. It is well known that
this condition induces
 a unique  almost complex structure $\widetilde J$ on 
$\mtilde$
such that the blow up map
$\pi: \mtilde\to \m$ is  $(J,\widetilde J)$-holomorphic. In particular
$\pi: \widetilde M\setminus E\to M\setminus \{x_0\}$ is a
biholomorphic map.

However note that is possible that the set 
$\mathcal{J}_\iota:=\{ J\in \mathcal{J}(M,\omega) | \iota^*J=J_0\}$
might not contain a regular almost complex structure. 
For the map $\iota^*: \mathcal{J}(M,\omega) \to \mathcal{J}(\iota (B^{2n}(\rho)),\omega)$ that
restricts an almost complex structure to the image of the embedded ball
is continuous. Hence $\mathcal{J}_\iota=(\iota^*)^{-1}(J_0)$ is closed.

\begin{lem}
Let $\iota: (B^{2n}(r), \omega_0)\to (M,\omega)$   be a symplectic embedding 
and $\mathcal{J}_\iota:=\{ J\in \mathcal{J}(M,\omega) | \iota^*J=J_0\}$.
Then $\mathcal{J}_\iota\subset \mathcal{J}(M,\omega) $ is closed and its complement is dense
\end{lem}

\section{Lagrangian submanifolds and holomorphic disks}
\label{s:Laganddisk}

\subsection{Lift of holomorphic disks}
Fix a symplectic embedding $\iota:  (B^{2n}(\rho),\omega_0)\to (M,\omega)$
and  set   $x_0:= \iota(0)$ to be the base point.
Consider $L\subset (M,\omega)$ a
Lagrangian submanifold
such that $L\cap \iota B^{2n}(\rho)$ is empty.
Then by part (2) of  Proposition
\ref{p:propblow} it follows that $\widetilde L:=\pi^{-1}(L)$
is a Lagrangian submanifold in $(\widetilde M,
\widetilde\omega_\rho)$. However if $L$ is such that $L\cap \iota B^{2n}(\rho)$ is not empty,
then $\widetilde L$
is not necessarily a Lagrangian submanifold, even if
$x_0\notin L$.


Let $J$  be a $\omega$-compatible almost complex structure on
$\m$  as in Section \ref{s:symplecticblow}
and $\widetilde J$ the unique $\widetilde\omega_\rho$-compatible
almost complex structure on $\mtilde$ such that the blow up map
$\pi: \m\to \mtilde$ is  $(J,\widetilde J)$-holomorphic. Hence
$\pi: \widetilde M\setminus E\to M\setminus \{x_0\}$ is a
biholomorphic map, therefore  any  $\widetilde J$-holomorphic disk $\widetilde u:
(D,\partial D)\to (\widetilde M,\widetilde L)$   projects to a $J$-holomorphic disk
$\pi\circ \widetilde u$ on $(M,L)$.
And vice versa, if $u$
is a $J$-holomorphic disc on $(M,L)$ such that
$x_0\notin u(D)$, then $\widetilde u:=
\pi^{-1}\circ u$ is a $\widetilde J$-holomorphic disc in $(\widetilde M,\widetilde L)$.
Recall that we are assuming that the Lagrangian submanifold $L\subset\m$
does not contain the base point $x_0$.

It only remains to analyse the  lift of the $J$-holomorphic disk $u: (D,\partial D)\to (M,L)$ in the
case when
$x_0\in u(D)$.
Since the base point $x_0$ is not on the Lagrangian submanifold
$L\subset \m$ and  the almost complex structure satisfies $\iota ^*J=J_0$,
in order to define the lift of $u$ to $\mtilde$  we ignore
the submanifold $L$ and  consider the case $(M,\omega)=(\mathbb{C}^n,\omega_0)$
blown up at the the origin. Thus $x_0=0$  and $u:D\to \mathbb{C}^n$ is holomorphic with respect to the
standard complex structure
and $0\in u(D)$.

For
let $u:D\to \mathbb{C}^n$ be a non constant holomorphic map such that
$u(0)=0$. 
Then each non constant component function  $u_j:D\to \mathbb{C}$  of $u$
can be written as $u_j=z^{k_j}h_j$ where $k_j$  is the order of the zero
of $u_j$ at $0\in D$. Thus  $k_j$  is a positive integer and
$h_j$ is holomorphic function that  does not vanish at $0$. Except if
 $u_j\equiv 0$, in this case we set  $k_j$ to be equal to $\infty.$
Thus the holomorphic map $u$ can be expressed as
\begin{eqnarray}
\label{e:orderlift}
u(z)=z^k (\hat h_1(z),\ldots,\hat h_n(z))
 \end{eqnarray}
where $k:=\min{k_j}$ and  at least one coordinate function $\hat h_j$
does not vanish at $0$
since we assumed that $u$ is non constant. 
The lift
$\widetilde u:D\to \widetilde{\mathbb{C}}^n$
of $u$ is defined as
\begin{eqnarray}
\label{e:ulift}
\widetilde u (z):  
=( u(z)  ,  [\hat h_1(z) :\cdots : \hat h_n(z)]).
\end{eqnarray}
So defined, $\widetilde u$ is holomorphic and projects to $u$ under the blow up map.
Notice that if $\widetilde u_0$ is another $\widetilde J$-holomorphic lift of $u$, then
$\widetilde u$ and $\widetilde u_0$ agree on $D\setminus \widetilde u^{-1}(E)$.  Since the maps
are $\widetilde J$-holomorphic they agree on all $D$, thus the holomorphic lift of $u$ is unique.


\begin{rema}
If $\psi:\mathbb{C}^n\to \mathbb{C}^n$ is a biholomorphic map  such that
$\psi(0)=0$ and $u:D\to \mathbb{C}^n$ is as above then the factorisation of
$\psi\circ u$   as in Eq. (\ref{e:orderlift}) gives the same value of $k$ as that of $u$. Thus $k$ is
 independent of the coordinate system.  
\end{rema}

\smallskip

Now that we have defined the lift  of $u:(D,\partial D)\to (M, L)$ to $\mtilde$, there is one more
consideration that needs  attention; the behaviour of $u$ at the blown up point $x_0$.
If  $u:(D,\partial D)\to (M, L)$ is a non constant $J$-holomorphic disk and
$z\in D\setminus \partial D$ is such that $u(z)=x_0$,  then
we define the {\em multiplicity
of $u$ at $z$}   has the integer $k\in \{0,1,\ldots ,\infty \}$
that appears in Eq. (\ref{e:orderlift}). Also we
define the {\em multiplicity  of $u$
at $x_0$} as $k_1+\cdots+k_r$ where $u^{-1}(x_0)=\{z_1,\ldots ,z_r\}$ and the multiplicity
of $u$ at $z_j$ is $k_j$. In the case when $x_0$ is not in the image of $u$,
we say that $u$
 has multiplicity zero at $x_0$.  Recall that since $u$ is $J$-holomorphic, the preimage of a point under $u$
is a finite set.


\begin{prop}
\label{p:lift}
Let $J$ be an almost complex structure on $(M,\omega)$ as above,  such that $\iota ^*J=J_0$
and $\widetilde J$ the unique almost complex structure on $\mtilde$
such that $\pi$ is $(\widetilde J, J)$-holomorphic.
If $ u: (D,\partial D)\to (M,L)$ a non constant  $J$-holomorphic disk,
 then there exists a unique
$\widetilde J$-holomorphic map
$\widetilde u: (D,\partial D)\to (\widetilde M,\widetilde L)$ such that
$\pi\circ \widetilde u=u$. Moreover if $u$ has multiplicity
$k$ at $x_0$, then $\widetilde u \cdot E=k$.
\end{prop}
\begin{proof}
It only remains to prove the relation  $\widetilde u \cdot E=k$.
Since the underlying manifold $\mtilde$ agrees with the complex blow up and
the exceptional divisor
$E$ is a $\widetilde J$-holomorphic submanifold of real codimension two of
$(\widetilde M,\widetilde \omega_\rho)$.
Now the multiplicity of $u$ at $x_0$ is $k$, therefore
 $\widetilde u \cdot E=k$.
\end{proof}

Recall from Theorem \ref{t:invariance} the invariance of the blown up point.
Thus, the  holomorphic  disks that will by considered in the framework
of LFH  would have $k=0$, that it the case when the blown up point
is irrelevant for the LFH; or $k=1$ that appears in the  presence of some form
of uniruled condition.

\subsection{Monotone Lagrangians on blow ups}
For a Lagrangian submanifold  $L$ of $(M,\omega)$, there
exist two classical morphisms
$$
I_{\mu,L}:\pi_2(M,L)\to \mathbb{Z}
\ \ \textup{ and } \ \ I_{\omega,L}:\pi_2(M,L)\to \mathbb{R};
$$
the Maslov index and symplectic area morphisms respectively.
A Lagrangian submanifold  is said to be {\em monotone} 
if there exists $\lambda>0$ such that
$
I_{\mu,L}=\lambda\cdot I_{\omega,L}.
$
The constant $\lambda$ is called the monotonicity constant of $L.$
As mentioned in Section
\ref{s:lfh }, in order to define Lagrangian Floer homology one restricts
to monotone Lagrangians. Thus if $L\subset (M,\omega)$
is a monotone Lagrangian submanifold
 we need to guarantee that $\widetilde L$
is a monotone Lagrangian on the one-point blow up
$(\widetilde M,\widetilde \omega_\rho)$.

If $L$ is a monotone Lagrangian submanifold of $(M,\omega)$ with monotonicity constant
$\lambda$, then $(M,\omega)$ is monotone symplectic.  That is
$I_c=(\lambda/2)I_\omega $.
In this case $I_\omega $
is defined on
 $\pi_2(M)$ and the morphism $I_c:\pi_2(M)\to\mathbb{Z}$
is given by evaluating at the first Chern class  of $(M,\omega)$ with respect
to any almost complex structure. Here $\alpha:=\lambda/2$
is the  monotonicity constant of $(M,\omega)$.

The first Chern classes  of $(M,\omega)$  and  the one-point blow
$\mtilde$ are related by the equation
\begin{eqnarray}
\label{e:chernrelation}
c_1(\widetilde M)=\pi^*(c_1(M)) - (n-1)\textup{PD}_{\widetilde M}(E),
\end{eqnarray}
where $E$ is the class of the exceptional divisor. 
Recall that if $L_E$ stands for  the complex line in the exceptional divisor in
$(\widetilde M,\widetilde \omega_\rho)$, then
$\widetilde \omega_\rho(L_E)=\pi \rho^2$ and $\textup{PD}_{\widetilde M}(E)([L_E])=-1.$
Since
the underlying manifold for the symplectic  blow up is independent of the weight it follows 
 from   Eq. (\ref{e:chernrelation}) the symplectic one-point blow
up   $(\widetilde M,\widetilde\omega_\rho)$ is  monotone if and only if
\begin{eqnarray}
 \label{e:monotoneweight}
\rho^2= { \frac{n-1}{\alpha\pi}}= { \frac{2(n-1)}{\lambda\pi}}.
\end{eqnarray}
Furthermore if Eq. (\ref{e:monotoneweight}) holds, then  $(\widetilde M,\widetilde\omega_\rho)$
and $(M,\omega)$ have
the same monotonicity constant.

Throughout the paper we make the following assumptions.
If the condition of monotonicity  on $(M,\omega)$ is required,
we will  assume that the  Gromov's width of $\m$ is greater than
$ {n-1}/{\alpha}$, where $\alpha$ is its monotonicity constant; and  the weight $\rho$
 of the one-point blow up $(\widetilde M,\widetilde\omega_\rho)$  is subject to Eq. (\ref{e:monotoneweight}).

The two  homotopy   long exact sequences of
the pairs $(\widetilde M,\widetilde L)$ and $(M,L)$ are related by the blow up map,
in the sense that the  diagram
\begin{center}
\begin{tabular}{ccccccccccc}
$\to$ &$\pi_2( \widetilde L )$&$ \overset{\widetilde i_*}{\to}$
 & $\pi_2( \widetilde M)$  & $ \overset{\widetilde j_*}{\to}$ &$\pi_2(\widetilde M, \widetilde L )$
  &$ \overset{\widetilde \delta}{\to}$  &$\pi_1( \widetilde L)$  & $\to$    &   $\pi_1(\widetilde M ) $ &$\to$\\
    & $\downarrow =$   &  &$\downarrow {\pi_*}$  &  &$\downarrow{\pi_*}$  &  &$\downarrow =$  &   
     & $\simeq \downarrow{\pi_*}$ &\\
$\to$ &$\pi_2( L )$& $\overset{i_*}{\to}$
  & $\pi_2( M)$  & $\overset{j_*}{\to}$  &$\pi_2(M, L )$  & $ \overset{\delta}{\to}$ &$\pi_1( L)$  & $\to$ 
  &  $ \pi_1( M ) $ &$\to$ \\
\end{tabular}
\end{center}
is commutative.


\begin{lem}
\label{l:pi2}
If $L$ is a Lagrangian submanifold in $\m$
then
the map $\pi_*:\pi_2(\widetilde M, \widetilde L )\to\pi_2( M, L )$ is surjective  and
$$
\textup{ker } \{\pi_*: \pi_2(\widetilde M, \widetilde L )\to\pi_2( M, L )\} = \widetilde j_*
\textup{ker } \{ \pi_*:\pi_2(\widetilde M )\to\pi_2( M )\}.
$$
\end{lem}
\begin{proof}
As smooth manifolds  $\widetilde M\simeq M\# \mathbb{C}P^n$. Therefore 
$\pi_*: \pi_2(\widetilde M)\to\pi_2(M)$ is surjective. 
Hence four of the vertical maps
in the above diagram are surjective, hence so is the map in the middle
$\pi_*:\pi_2(\widetilde M, \widetilde L )\to\pi_2( M, L )$.


Now we  show that the kernel of
$\pi_*:\pi_2(\widetilde M, \widetilde L )\to\pi_2( M, L )$ is contained in
$\widetilde j_*\textup{ker } \{\pi_*: \pi_2(\widetilde M )\to\pi_2( M )\}$; the reverse inclusion follows by the
commutativity of the diagram.

For, let $u\in \pi_2(\widetilde M, \widetilde L )$ be an element
 that maps to the identity element $e\in \pi_2( M, L )$. Since $\pi_1(L)=\pi_1(\widetilde L)$
 then $\widetilde \delta (u)=e$ and
by exactness of the top sequence
  there
 is $w\in \pi_2(\widetilde M)$ such that $\widetilde j_*(w)=u$.
Note that $\pi_*(\widetilde j_*(w))=e$. Thus by exactness and the fact that
  $\pi_2(L)=\pi_2(\widetilde L)$, there is $w^\prime\in \pi_2(\widetilde L)$ such that
$\pi_*(\widetilde i_*(w^\prime))=\pi_*(w)$.
Therefore  $\widetilde j_*(w-\widetilde i_*(w^\prime))=u$  and
$w-\widetilde i_*(w^\prime)$ maps to $e\in\pi_2(M).$
\end{proof}

If follows from Lemma \ref{l:pi2} that
elements of the kernel of $\pi_*:\pi_2(\widetilde M, \widetilde L )\to\pi_2( M, L )$
are induced by absolute classes. 
Therefore any $[\widetilde u] \in \pi_2(\widetilde M, \widetilde L )$
can be  expressed as $[\widetilde u]=[u_0\#w]$ where $[w]\in
\textup{im }\{  \pi_2(\widetilde M)\to \pi_2(\widetilde M, \widetilde L )\}$
and $u_0:(D,\partial D)\to (\widetilde M, \widetilde L )$ does not intersect
the exceptional divisor. 
\begin{lem}
\label{l:factor}
If $L$ is a  Lagrangian submanifold in $\m$
  and $[u]\in \pi_2(\widetilde M,\widetilde L)$,
then there exist  $\ell\in\mathbb{Z}$
and $u_0:(D,\partial D)\to (\widetilde M, \widetilde L )$
 such that 
$u_0$ does not intersect the exceptional divisor and
$$
[u]=[u_0\# \widetilde j(\ell [L_E])].
$$
\end{lem}
\begin{proof}
First assume that  $[u]$ is such that $\pi_*[u]=e$. Then by Lemma \ref{l:pi2}
there exists $[w_0]\in \pi_2(\widetilde M)$
such that $[u]=\widetilde j_*[w_0]=[\widetilde j\circ w_0]$. 
Let $[v]:=\pi_*[w_0]$, thus $j_*[v]=e$. Then  by the commutativity of the above diagram
and the fact that $\pi_2(\widetilde L)=\pi_2(L)$, there is $[v_0]\in \pi_2(\widetilde L)$
such that $\pi_*\circ \widetilde i_*[v_0]=[v]$. Therefore $[w_0]-\widetilde i_*[v_0]\in \pi_2(\widetilde M)$
is such that $\widetilde j_*( [w_0]-\widetilde i_*[v_0] )=[u]$ and maps to $e$ under $\pi_*$.
Thus the result holds in this case.

Now for arbitrary $[u]\in\pi_2(\widetilde M,\widetilde L)$, let
$[u_0]:=\pi_*[u]\in\pi_2( M, L)$. Since $x_0$ is not in $L$, there exist a continuous map
$u_0^\prime: (D,\partial D)\to (M,L)$ such that $x_0 $ is not in $u_0^\prime(D)$
and $[u_0]=[u_0^\prime]$. In particular, the map $u_0^\prime$ lifts to a map 
$(\widetilde M,\widetilde L)$  that does not intersect the exceptional divisor. Let
$\widetilde u_0$ be such a map, thus $[\widetilde u_0]\in\pi_2(\widetilde M,\widetilde L)$ and 
$[\pi\circ\widetilde u_0]=[u_0^\prime]$.
Notice that $[u]-[\widetilde u_0]$ maps to $e$ under $\pi_*$. Hence
there exists $[w]\in \pi_2(\widetilde M)$
such that $\pi_*[w]=e$ and
$$
[u]=[\widetilde u_0]+\widetilde j_*[w]=[\widetilde u_0\# (\widetilde j\circ w)].
$$
Further since  $\pi_*[w]=e$ then as a homology class
$[w]$ is a purely exceptional class. This means that $[w]=\ell [L_E]$
for some $\ell\in\mathbb{Z}$. Therefore $[u]=[u_0\# \widetilde j(\ell [L_E])].$
\end{proof}

With these results is now possible to show that the lift to the one-point blow up
of a monotone Lagrangian submanifold is also monotone.

\begin{lem}
\label{l:maslovindexrelation}
Let $L$ be a Lagrangian submanifold in $\m$.
If $\widetilde u:(D, \partial D)\to (\widetilde M, \widetilde L)$
is  a smooth map,
then
$$
\mu_{\widetilde L}[\widetilde u] =\mu_{L}[\pi\circ\widetilde u]+2(n-1)\ell
$$
for some $\ell\in\mathbb{Z}.$
\end{lem}

Next, we see that the condition of monotonicity of a Lagrangian submanifold is preserved
by the proper transform.

\begin{lem}
\label{l:monotonelag}
Let $L$ be a  Lagrangian submanifold in $\m$,
 then
  $\widetilde L$ is monotone Lagrangian submanifold of $(\widetilde M,\widetilde \omega_\rho)$
with the same monotonicity constant as $L$.
\end{lem}

\begin{proof}
Let $\lambda$ and $\alpha:=\lambda/2$
 be the monotonicity constants  of $L\subset(M,\omega)$ and $(M,\omega)$
respectively. Recall that the value of $\rho =\sqrt{{ \frac{n-1}{\alpha\pi}}}$ is such that $(\widetilde M,\widetilde \omega_\rho)$
is monotone with monotonicity constant $\alpha.$

For  $[\widetilde u]$ in $\pi_2(\widetilde M,\widetilde L)$, by Lemma \ref{l:factor} we have that
 $[\widetilde u]=[u_0  \#\widetilde j (\ell[L_E])]$ where
$u_0: (D,\partial D)\to (\widetilde M,\widetilde L)$ does not intersect
the exceptional divisor and $\ell\in \mathbb{Z}$. 
Therefore
$$
I_{\widetilde\mu,\widetilde L}([\widetilde u])= I_{\widetilde\mu,\widetilde L}([u_0])+ 2c_1(\widetilde M)(\ell[L_E]).
$$
Since $u_0$ does not intersects the exceptional divisor, we can assume that   its image lies 
in $\widetilde M\setminus \pi^{-1}(\iota B^{2n}(\rho))$. By Proposition 
\ref{p:propblow}, 
$(\widetilde M\setminus \pi^{-1}(\iota B^{2n}(\rho)),\widetilde\omega_\rho)$
is symplectomorphic to 
$( M\setminus\iota B^{2n}(\rho),\omega)$ under the blow up map. 
The Maslov index is invariant under symplectic diffeomorphisms, thus
$
I_{\widetilde\mu,\widetilde L}([u_0])= I_{\mu,L}([\pi\circ u_0])
$
and
 $$
I_{\widetilde\mu,\widetilde L}([\widetilde u])=  I_{\mu,L}([\pi\circ u_0])+ 2\ell c_1(\widetilde M)([L_E]).
$$

Recall that $\widetilde\omega_\rho([L_E])=\pi \rho^2$ and
\begin{eqnarray*}
c_1(\widetilde M)[L_E]&=& \pi^*(c_1(M))[L_E] - (n-1)\textup{PD}_{\widetilde M}(E)[L_E]
=n-1.
\end{eqnarray*}
Thus
\begin{eqnarray*}
 I_{\mu,L}([\pi\circ  u_0 ])+ 2\ell c_1(\widetilde M)([L_E]) &= &
I_{\mu,L}([\pi\circ  u_0])+2(n-1)\ell \\
 &= &\lambda I_{\omega}([\pi\circ u_0])+  2(n-1)\ell \\
&= &\lambda I_{\omega}([\pi\circ u_0])+ (2\alpha) \pi \rho^2 \ell \\
&= &\lambda \omega([\pi\circ  u_0])+ \lambda   I_{\widetilde\omega_\rho}([w]) =
\lambda I_{\widetilde\omega_\rho}([\widetilde u]).
\end{eqnarray*}
That is, $  I_{\widetilde\mu,\widetilde L}([\widetilde u])   =\lambda I_{\widetilde\omega_\rho}([\widetilde u])$
for all $[\widetilde u]\in \pi_2(\widetilde M,\widetilde L)$
and $\widetilde L$ is a monotone Lagrangian submanifold. 
\end{proof}


Notice  from the above proof, that  the  Maslov index of
$\widetilde u:(D, \partial D)\to (\widetilde M, \widetilde L)$ can be written in terms of
the Maslov index of
$\pi\circ \widetilde u$.
In particular, in the case of a $J$-holomorphic disk, we have a precise description of
the integer $\ell$ that appears in the above formula.

\begin{prop}
\label{p:maslovindx}
Let $(M,\omega)$ and  $L$ as in Lemma \ref{l:monotonelag} and $J$ a $\omega$-compatible almost
complex structure on $(M,\omega)$. If $\widetilde u: (D,\partial D)\to (\widetilde M,\widetilde L)$  is
 $\widetilde J$-holomorphic  and
$[\widetilde u]\cdot E=\ell\geq 0$, then
$$
\mu_{\widetilde L}[\widetilde u] =\mu_{L}[\pi\circ\widetilde u] -2(n-1)\ell.
$$
\end{prop}

Hence the holomorphic disks $\widetilde u$ and $\pi\circ\widetilde u$ have the same
Maslov index if and only if  $\widetilde u$ does not intersect the exceptional divisor,  
or equivalently $\pi\circ\widetilde u$ does not contain the base point $x_0$.
 
Finally recall that  the minimal Maslov number $N_L$  of a Lagrangian submanifold $L$ in $\m$
 is defined as the the positive generator of the image of 
$I_{\mu, L}:\pi_2(M,L)\to \mathbb{Z}$. Then under the considerations of 
Lemma \ref{l:maslovindexrelation},
 $
N_{\widetilde L}\leq
N_L .
$

\begin{rema}
The statements presented in this section regarding the Maslov index $\mu_L$
of a Lagrangian submanifold $L$, also apply to the relative Maslov index
$\mu_{L_0,L_1}$ of the pair of Lagrangian submanifolds $L_0$ and $L_1$.
\end{rema}

\section{Lagrangian Floer homology on the blow up}

Let $\m$ be a symplectic manifold that is
  either closed
or convex at infinity,
and $L_0$ and $L_1$   closed Lagrangian submanifolds 
 that intersect
transversely and $N_{L_j}\geq 3$ for $j=0,1$. 
For the moment, the Lagrangian submanifolds do not have to be monotone.
As above, we assume that $L_0$ and $L_1$   do not intersect the image
of the embedded ball.
Finally we also assume that there
exists a $\omega$-compatible  almost complex structure $J$ in $\mathcal{J}_\textup{reg}(L_0,L_1)$
such that $\iota^*J=J_0$.

\begin{prop}
\label{p:regular}
Let $\m, L_0,L_1$ and $\iota:(B^{2n}(\rho),\omega_0)\to \m$ as above.
If  $J$ is a regular $\omega$-compatible almost complex structure  for  $(L_0,L_1)$ and
$\widetilde J$ the unique $\widetilde\omega_\rho$-compatible almost complex structure on $\mtilde$ such
that $\pi$ is $(\widetilde J, J)$-holomorphic, then $\widetilde J$
is regular for $(\widetilde L_0,\widetilde L_1)$.
\end{prop}
\begin{proof}
Let  $\widetilde u:
(D,\partial D)\to (\widetilde M,\widetilde L_0\cup \widetilde L_1)$
be a $\widetilde J$-holomorphic disk that joints the intersection points
 $\widetilde  p$ and $\widetilde q$. Since the  blow up map is holomorphic,
$\pi\circ \widetilde u$
is a $J$-holomorphic disk that joints the intersection points
$p=\pi(\widetilde  p)$ and $q=\pi(\widetilde q)$ and its boundary lies in $L_0\cup L_1$,
$\pi\circ \widetilde u(\cdot,j)\in L_j$ for $j=1,2$.
Since $J$ is regular for $(L_0,L_1)$,
 then the  operator
$$
D_{\overline\partial(J),\pi\circ \widetilde u}: W_k^p((\pi\circ \widetilde u)^*TM;L_0,L_1)
\to L_p((\pi\circ \widetilde u)^*TM)
$$
is surjective.

The blow up map induces an operator between the spaces of sections 
$L_p(\widetilde u^*T\widetilde M)$ and
$L_p((\pi\circ \widetilde u)^*TM)$ as follows. In the case when
$\widetilde u$ does not intersect the exceptional divisor, the map
$$	
\pi_{\widetilde u}^L: L_p(\widetilde u^*T\widetilde M)\to L_p((\pi\circ \widetilde u)^*TM)
$$
is defined  as $\pi_{\widetilde u}^L(\xi):=\pi_*(\xi)$. Note that it 
  is well defined and surjective.
Now in the case when $\widetilde u(D)\cap E$ is not empty, then since
$\widetilde u$ is holomorphic we have that $\widetilde u^{-1}(E)$
is a finite set in $D$. So in this case $\pi_{\widetilde u}^L(\xi)$ is defined in the
same way as in the previous case on $D\setminus  \widetilde u^{-1}(E)$ and equal to zero
on $  \widetilde u^{-1}(E)$.  Also in this
case $\pi_{\widetilde u}^L$ is well defined and surjective.  That is, for every
${\widetilde u}$
holomorphic disk the map $\pi_{\widetilde u}^L$ is surjective.
The same reasoning shows that the map
$$
\pi_{\widetilde u}^W:
W_p^k(\widetilde u^*T\widetilde M;\widetilde L_0,\widetilde L_1)\to
W_p^k((\pi\circ \widetilde u)^*TM;L_0,L_1)
$$
defined as $\pi_{\widetilde u}^W(\xi)=\pi_*(\xi)$ on $D\setminus  \widetilde u^{-1}(E)$
and zero on $  \widetilde u^{-1}(E)$ is well defined and surjective.

Notice that we have a commutative relation
$$
\pi_{\widetilde u}^L\circ D_{\overline\partial(\widetilde J),\widetilde u}
=D_{\overline\partial(J),\pi\circ \widetilde u}\circ \pi_{\widetilde u}^W.
$$
Since $D_{\overline\partial(J),\pi\circ \widetilde u}$ is surjective, 
then $D_{\overline\partial(\widetilde J),\widetilde u}$ is surjective and $\widetilde J$
is regular for $(\widetilde L_0,\widetilde L_1)$.
\end{proof}

A similar argument applies to the case of regularity of the almost complex
structure in the case of the pearl complex.

For
 $\widetilde  p$ and $\widetilde q$ in $\widetilde L_0\cap\widetilde L_1$,
 $J $ a regular $\omega$-compatible almost complex structure on $\m$ and $\beta\in \pi_2(\widetilde M,
\widetilde L_0\cup\widetilde L_1)$ there is a smooth map
$$
\mathcal{M}\pi:\mathcal{M}(\widetilde p,\widetilde q,  \beta,\widetilde  J) \to
\mathcal{M}( p,q,  \pi_*(\beta), J)
$$
induced by the blow up map.  This map is not  necessarily surjective.
For, suppose that $u$ a $J$-holomorphic disk such that $[u]=\pi_*(\beta)$,
$\pi_*(\beta)\in\pi_2(M,L_0\cup L_1)$ is non trivial 
and $u(z_0)=x_0$ for some $z_0\in \textup{Int}(D)$. Then by
Propositions \ref{p:lift} and \ref{p:maslovindx},
$u$ has the unique holomorphic  lift $\widetilde u$ is such that
$\mu_{\widetilde L_0,\widetilde L_1}(  \widetilde u)<\mu_{ L_0,L_1}(u)$.
Hence if the class $\beta$ does not have an  exceptional part, we get that  the lift $\widetilde u$
does not lie in $\mathcal{M}(\widetilde p,\widetilde q,  \beta,\widetilde  J)$.

However if we ignore the homotopy class and consider the whole moduli space, then by
Propositions \ref{p:lift}  and \ref{p:regular}, the map $\mathcal{M}\pi$ is surjective.
\begin{prop}
Let $L_1$and $L_2$ as in Proposition \ref{p:regular}. 
The map
 $$
\mathcal{M}\pi:\mathcal{M}(\widetilde p,\widetilde q, \widetilde  J) \to
\mathcal{M}( p,q,  J)
$$
given by $\mathcal{M}\pi(\widetilde u)=\pi\circ \widetilde u$
is surjective.
\end{prop}

Next we give the proof of some of the results that where stated at the Introduction which are
related to Lagrangian Floer homology.

\noindent{\bf Theorem \ref{t:invariance}.}{\em \, 
Let $L_0,L_1\subset (M,\omega)$  be admissible  Lagrangian submanifolds.
Assume that there exists 
$J\in \mathcal{J}_\textup{reg}(L_0,L_1)$ 
such that  $\iota^*J=J_0$ where $\iota:(B^{2n}(\rho),\omega_0)\to (M,\omega)$ is 
a symplectic embedding that avoids the
Lagrangian submanifolds and $\rho$ is given by Eq. (\ref{e:rho}).
If
 $\pi_j: (\widetilde M_j,\widetilde\omega_j)\to (M,\omega)$
are  the  monotone one-point blow up of
$(M,\omega)$ at $x_j\in M\setminus (L_0\cup L_1)$ for $j=1,2$,
then
\begin{eqnarray*}
\textup{HF}_*(\pi_1^{-1}(L_0),\pi_1^{-1}(L_1))\simeq
\textup{HF}_*(\pi_2^{-1}(L_0),\pi_2^{-1}(L_1))
\end{eqnarray*}
as $\Lambda$-modules.
}
\begin{proof} 
The one-point blow up is defined on symplectic manifolds of dimension greater than two.
Therefore there exists a path connecting $x_1$ to $x_2$
that does not intersect the
Lagrangian submanifolds.
 Hence there is a Hamiltonian diffeomorphism $\psi$
on $(M,\omega)$ such that $\psi(x_1)=x_2$ and its support does
not intersects $L_0\cup L_1$.

Since the image of $\iota:(B^{2n}(\rho),\omega_0)\to (M,\omega)$  misses the Lagrangian submanifolds,
then $\psi\circ \iota_1$ is also a symplectic embedding of the ball 
such that  $\psi\circ\iota(0)=x_2$ and its image also misses the Lagrangians.
For $j=1,2$, let $(\widetilde M_j,\widetilde\omega_{j})$ be the corresponding
symplectic one-point
blow up at $x_j$. 
Note that $(\psi^{-1})^*(J)$ is  regular and $(\psi\circ\iota)^*\circ(\psi^{-1})^*(J)=J_0$.
Then by Proposition \ref{p:regular},
 $J$ and $(\psi^{-1})^*(J)$ induced regular almost complex structure on the
their respective one-point blow up.  Moreover, $J$-holomorphic strips
that go thru $x_1$ are mapped to  $(\psi^{-1})^*(J)$-holomorphic strips
that go thru $x_2$ preserving the Maslov-Viterbo index and vice versa.
Hence $\textup{HF}_*(\pi_1^{-1}(L_0),\pi_1^{-1}(L_1))$
in $(\widetilde M_1,\widetilde\omega_{1})$ is isomorphic
as $\Lambda$-modules to $\textup{HF}_*(\pi_2^{-1}(L_0),\pi_2^{-1}(L_1))$
in $(\widetilde M_2,\widetilde\omega_{2})$.
\end{proof}

The Floer differential $\partial:  \textup{CF}(L_0,L_1)\to  \textup{CF}(L_0,L_1)$
only involves holomorphic strips of Maslov-Viterbo index $1$.
From Proposition \ref{p:maslovindx}, we know that
a holomorphic strip of Maslov-Viterbo index 1 in $\m$
lifts also to a holomorphic strip of 
index 1 in $(\widetilde M,\widetilde\omega_\rho)$.
But also, from  Proposition \ref{p:maslovindx},
a holomorphic strip in $(M,\omega)$ of Maslov-Viterbo index $2n-1$ 
lift also to a holomorphic strip of Maslov-Viterbo index $1$.
Further if $J$ is regular for $(L_0,L_1)$, from
Proposition \ref{p:regular} we know that $\widetilde J$ is also regular.
Therefore if on $M\setminus (L_0\cup L_1)$ there exists a point that is not
contained in a $J$-holomorphic strip of Maslov-Viterbo index $2n-1$, then
$$
\textup{HF}_*(\widetilde L_0,\widetilde L_1)\simeq
\textup{HF}_*(L_0,L_1).
$$
This arguments proves the following proposition 
of the Introduction.

\noindent{\bf Theorem \ref{t:noholo}.}
{\em 
Let $L_0,L_1\subset (M,\omega)$  be admissible Lagrangians.
If there exists $x_0\in M\setminus (L_0\cup L_1)$ and
 $J\in\mathcal{J}_{\textup{reg}}(L_0,L_1)$
such that $\iota^*J=J_0$; where $\iota$ is the symplectic embedding of the ball,
 and $x_0$ does not lie
in any $J$-holomorphic strip of Maslov-Viterbo index $2n-1$, then
\begin{eqnarray*}
\textup{HF}_*(\widetilde L_0,\widetilde L_1)\simeq
\textup{HF}_*(L_0,L_1).
\end{eqnarray*}
}

\section{Proof of Theorem \ref{t:dif} and Computation of $\textup{QH}_*(
\widetilde{\mathbb{T}}_\textup{Cliff})$}
\label{s:clif}

First we  give a proof of Theorem \ref{t:dif}.

\noindent{\bf Theorem \ref{t:dif}.}
{\em
Let $L$ be an admissible Lagrangian in $(M,\omega)$ that is
$(1,2)$-uniruled of order $2$, $x_0\in M\setminus L$ a generic point and suppose that $\textup{dim}(M)=4$. 
If $p$ and $q$ in $L$ are critical points of a generic Morse-Smale function $f$
with respect to a generic Riemannian metric $g$ on $L$ such that $\textup{ind}(p)-\textup{ind}(q)-1=-2$, then
$$
\langle \widetilde d(p), q\rangle=  \langle  d(p), q\rangle +_{\mathbb{Z}_2} k
$$
where $k$ is the number, mod $2$, of classes $A\in \textup{H}_2(M,L;\mathbb{Z})$ such that for some 
$J\in \mathcal{J}_\textup{reg}(M,L)$,
 such that $\iota^*J=J_0$ where $\iota$ is the symplectic embedding of the ball,
the moduli space of pearly trajectories
$ \mathcal{P}(p,q,A;g,f,J)$ is non empty,
 $\mu_L(A)=4$ and there is a $J$-holomorphic disk $u$ such that
 $x_0\in u(D) $ and $A=[u]$.
}
\begin{proof}
Since the Lagrangian is admissible, there is an embedding $\iota: (B^4(\rho),\omega_0)\to (M\setminus L,\omega)$
such that $\iota(0)=x_0$ and $\iota^*J=J_0$ for some $J\in\mathcal{J}_\textup{reg}(M, L)$.
Further by Proposition \ref{p:regular}, the induced almost complex structure $\widetilde J$
in $(\widetilde M, \widetilde \omega_\rho)$ is regular. Thus $\widetilde J$-holomorphic
disk with boundary in $\widetilde L$ map under the blow up map to 
 $J$-holomorphic
disk with boundary in $ L$, and vice versa.

In order to compute $\langle \widetilde d(p), q\rangle$ where
$p$ and $q$ in $L=\widetilde L\subset (\widetilde M, \widetilde \omega_\rho)$
are given as in the statement, we must count all pearly trajectories
in $(\widetilde M, \widetilde \omega_\rho)$ that go from $p$ to $q$ that
contain a single $\widetilde J$-holomorphic disk of Maslov index $2$.
By Proposition \ref{p:maslovindx}, a  pearly trajectory
in $(M, \omega)$ that goes from $p$ to $q$ that
contains a single $J$-holomorphic disk of Maslov index $2$ lifts
to a desire pearly trajectory
in $(\widetilde M, \widetilde \omega_\rho)$. Note that since the 
Maslov index is two, generically the image of the $(D,\partial D)$
under the whole  collection of $J$-maps in such a class is 3-dimensional in $M$. Thus, generically it
avoids the generic point $x_0$. 
Moreover if the pearly
trajectory contains a $J$-holomorphic disk in the class $A\in \textup{H}_2(M,L)$,
the lift contains a $\widetilde J$-holomorphic disk in the class $\widetilde A\in \textup{H}_2(\widetilde M,
\widetilde L)$.

Again by Proposition \ref{p:maslovindx}, it follows that a pearly trajectory
in $(M, \omega)$ that goes from $p$ to $q$ that
contains a single $J$-holomorphic disk of Maslov index $4$, say in the class $A$,
and moreover there is a $J$-holomorphic disk such that $x_0\in u(D)$
and $A=[u]$, lifts 
to a desire pearly trajectory
in $(\widetilde M, \widetilde \omega_\rho)$. 
That is, to a  pearly trajectory
in $(\widetilde M, \widetilde \omega_\rho)$ from $p$ to $q$ that
contains a single $\widetilde J$-holomorphic disk of Maslov index $
2$ in the class $\widetilde A-L_E.$ 

Note that this are the only possibilities of pearly trajectories in $(\widetilde M, \widetilde \omega_\rho)$.
Hence the theorem follows. 
\end{proof}


As an application of our results we compute the pearly differential of
$\widetilde{\mathbb{T}}_\textup{Cliff}\subset  (\widetilde{\mathbb{C}P}\,^2,\widetilde\omega_\rho)$.
Recall from the work of Fukaya,  Oh, Ohta and  Ono  
\cite{FO3-toricI}; and  Entov and Polterovich in
 \cite{entov-pol-rigid}, that $\widetilde{\mathbb{T}}_\textup{Cliff}$
 is a wide Lagrangian. We present an alternative approach of this fact.

\noindent{\bf Theorem \ref{t:cliff}.}
{\em
After blowing up one point the proper transform of the Clifford torus $\mathbb{T}_\textup{Cliff}\subset
(\mathbb{C}P^2,\omega_\textup{FS})$
in $(\widetilde{\mathbb{C}P}\,^2,\widetilde\omega_\rho)$ is also a wide Lagrangian.
}
\begin{proof}
Assume that $\mathbb{T}_\textup{Cliff}$ is equipped with a generic Riemannian
structure and a Morse-Smale function $f$ with four critical points $p_0,p_1^a,p_1^b$ and $p_2$.
Here the index notation in the critical point indicates the Morse index of the critical point. 
It is know from  \cite{biran-cornea-rigidity}  that the standard complex structure $J$ on 
$(\mathbb{C}P^2,\omega_\textup{FS})$
is regular and
that the pearl differential $d$ is identically zero.

Further in order to compute the pearl differential $\widetilde d$ on the one-point blow up,
we need to know the pearly trajectories that contribute to $d$. 
To that end, consider the generators $A_0, A_1,A_2\in \textup{H}_2(\mathbb{C}P^2,
\mathbb{T}_\textup{Cliff};\mathbb{Z})$  such that $A_1$ and $A_2$ when restricted to 
$\mathbb{T}_\textup{Cliff}$ go to $c_1$ and $c_2$ respectively and  generate
$\textup{H}_1(
\mathbb{T}_\textup{Cliff};\mathbb{Z})$.
Thus, $A_0+A_1+A_2$ represents the absolute class of the line in
$\textup{H}_2(\mathbb{C}P^2,
\mathbb{T}_\textup{Cliff};\mathbb{Z})$
and 
$\mu_{\mathbb{T}_\textup{Cliff}}(A_0)=\mu_{\mathbb{T}_\textup{Cliff}}(A_1)=
\mu_{\mathbb{T}_\textup{Cliff}}(A_2)=2$.

According to the dimension of the moduli space of pearly trajectories there are exactly four
instances  where the pearly trajectories are formed by non constant $J$-holomorphic disks.  
Namely
$\langle   d(p_0),p_1^a  \rangle$, $\langle   d(p_0),p_1^b  \rangle$, 
$\langle   d(p_1^a), p_2  \rangle$ and $\langle   d(p_1^b), p_2  \rangle$.
Moreover, in each case there are exactly two 
the pearly trajectories each  containing a single $J$-holomorphic disk of Maslov index 2. 
For example, there are two pearly trajectories from $p_0$ to $p_1^a$, $\langle   d(p_0),p_1^a  \rangle=0$.
Both trajectories start with a $J$-holomorphic disk thru $p_0$  follow by a flow line that goes to
$p_1^a$.  In one case the $J$-holomorphic disk lies in the class $A_0$
and in the other case we can assumed to be $A_1$. That is we assume that the class
$c_1\in  \textup{H}_1( \mathbb{T}_\textup{Cliff};\mathbb{Z})$
intersects the unstable submanifold at $p_1^a.$
 For the other three cases
the pearly trajectories that appear are similar to the case the was described.

Additionally let $x_0=[1:a:b]$  be a generic point in $\mathbb{C}P^2\setminus\mathbb{T}_\textup{Cliff}$ 
such that $0<|a|<|b|<1.$ Hence $\widetilde{\mathbb{C}P}\,^2$ will stand for the blow up
of $\mathbb{C}P^2$ at $x_0$.  Moreover 
by \cite[Cor. 1.2.5]{biran-cornea-rigidity}
there is a symplectic embedding of $(B^4(1/\sqrt{3}),\omega_0)$
into 
$(\widetilde{\mathbb{C}P}\,^2\setminus \widetilde{\mathbb{T}}_\textup{Cliff},\omega_{\textup{FS}})$. 
Hence  by Lemmas \ref{l:monotonelag} and \ref{l:maslovindexrelation},
$\widetilde{\mathbb{T}}_\textup{Cliff}\subset  (\widetilde{\mathbb{C}P}\,^2,\widetilde\omega_\rho)$
is a monotone Lagrangian submanifold with minimal Maslov number equal to 2.
Henceforth, the Lagrangian quantum homology of 
$\widetilde{\mathbb{T}}_\textup{Cliff}\subset  (\widetilde{\mathbb{C}P}\,^2,\widetilde\omega_\rho)$
is well defined.

Recall that $(\mathbb{C}P^2,\omega_\textup{FS})$ is $(1,1)$-uniruled of order 4. Hence we cannot conclude
directly that $\textup{QH}_*({\mathbb{T}}_\textup{Cliff})$ and $\textup{QH}_*(
\widetilde{\mathbb{T}}_\textup{Cliff})$ are isomorphic.
In order to compute $\widetilde d$ and $\textup{QH}_*(\widetilde{\mathbb{T}}_\textup{Cliff})$ we use the same
Riemannian structure and Morse-Smale function on the Lagrangian submanifold.
Since $J$ is a complex structure on $\mathbb{C}P^2$, then it induces a complex structure $\widetilde J$
on 
$\widetilde{\mathbb{C}P}\,^2$. Further, by 
Proposition \ref{p:regular}, 
$\widetilde J$ is regular for $(\widetilde{\mathbb{C}P}\,^2,
 \widetilde{\mathbb{T}}_\textup{Cliff})$.

Now we compute the pearly trajectories in $\widetilde{\mathbb{C}P}\,^2$ that connect 
$p_0$ to $p_1^a$ and show that $\langle   \widetilde d(p_0),p_1^a  \rangle=0$. 
More precisely, the value of $k$ of Theorem \ref{t:dif} is two.
The two pearly
trajectories described above still exist; each having $\widetilde J$-holomorphic disks in the
classes $\widetilde B$ and $\widetilde A_1$. Note that $\mu_{\widetilde{\mathbb{T}}_\textup{Cliff}}
(\widetilde B)=\mu_{\widetilde{\mathbb{T}}_\textup{Cliff}}
(\widetilde A_1)=2$. 
However there are two new pearly trajectories that are induced by $J$-holomorphic
disks in $(\mathbb{C}P^2,
\mathbb{T}_\textup{Cliff})$ with Maslov index 4. For this, we rely on the classification of 
$J$-holomorphic disks on $(\mathbb{C}P^2,
\mathbb{T}_\textup{Cliff})$ given by Cho in \cite[Theorem 10.1]{cho-floer}.
There are six classes of $J$-holomorphic disk with Maslov index 4,
$A_j+A_k$ for $j,k\in \{0,1,2\}$. Of those
only the classes $A_1+A_2$ and $A_0+A_2$ contain $J$-holomorphic
disks that go thru $x_0$ and when restricted to $\mathbb{T}_\textup{Cliff}$
intersect the unstable submanifold at $p_1^a$.
Hence the new pearly trajectories that contribute to $\langle   \widetilde d(p_0),p_1^a  \rangle$
star with
a $\widetilde J$-holomorphic disk thru $p_0$ in the class
$\widetilde A_1+\widetilde A_2-L_E$ and a flow line to $p_1^a$; 
the other also stars with a $\widetilde J$-holomorphic
disk thru $p_0$ in the class $\widetilde A_0+\widetilde A_2-L_E$ and a flow line to $p_1^a$.
Hence $\langle   \widetilde d(p_0),p_1^a  \rangle=0$.

A similar argument applies to the other cases to conclude that
 $\langle  \widetilde  d(p_0),p_1^b  \rangle= 
\langle  \widetilde  d(p_1^a), p_2  \rangle =\langle   \widetilde d(p_1^b), p_2  \rangle=0$.
Hence $\textup{QH}_*(
\widetilde{\mathbb{T}}_\textup{Cliff})=\textup{H}_*(
\widetilde{\mathbb{T}}_\textup{Cliff};\mathbb{Z}_2)\otimes \Lambda$ 
and $\widetilde{\mathbb{T}}_\textup{Cliff}$ is a
wide Lagrangian in $(\widetilde{\mathbb{C}P}\,^2,\widetilde\omega_\rho)$.
\end{proof}

\bibliographystyle{acm}
\bibliography{/Users/andres/Dropbox/Documentostex/Ref.bib} 

\begin{thebibliography}{10}

\bibitem{auroux-a}
{\sc Auroux, D.}
\newblock A beginner's introduction to {F}ukaya categories.
\newblock In {\em Contact and symplectic topology}, vol.~26 of {\em Bolyai Soc.
  Math. Stud.} J\'anos Bolyai Math. Soc., Budapest, 2014, pp.~85--136.

\bibitem{biran-lagrangianbarriers}
{\sc Biran, P.}
\newblock Lagrangian barriers and symplectic embeddings.
\newblock {\em Geom. Funct. Anal. 11}, 3 (2001), 407--464.

\bibitem{biran-octav-lagquan}
{\sc Biran, P., and Cornea, O.}
\newblock A {L}agrangian quantum homology.
\newblock In {\em New perspectives and challenges in symplectic field theory},
  vol.~49 of {\em CRM Proc. Lecture Notes}. Amer. Math. Soc., Providence, RI,
  2009, pp.~1--44.

\bibitem{biran-cornea-rigidity}
{\sc Biran, P., and Cornea, O.}
\newblock Rigidity and uniruling for {L}agrangian submanifolds.
\newblock {\em Geom. Topol. 13}, 5 (2009), 2881--2989.

\bibitem{cho-floer}
{\sc Cho, C.-H.}
\newblock Holomorphic discs, spin structures, and {F}loer cohomology of the
  {C}lifford torus.
\newblock {\em Int. Math. Res. Not.}, 35 (2004), 1803--1843.

\bibitem{entov-pol-rigid}
{\sc Entov, M., and Polterovich, L.}
\newblock Rigid subsets of symplectic manifolds.
\newblock {\em Compos. Math. 145}, 3 (2009), 773--826.

\bibitem{floer-morse}
{\sc Floer, A.}
\newblock Morse theory for {L}agrangian intersections.
\newblock {\em J. Differential Geom. 28}, 3 (1988), 513--547.

\bibitem{floer-witten}
{\sc Floer, A.}
\newblock Witten's complex and infinite-dimensional {M}orse theory.
\newblock {\em J. Differential Geom. 30}, 1 (1989), 207--221.

\bibitem{fooo}
{\sc Fukaya, K., Oh, Y.-G., Ohta, H., and Ono, K.}
\newblock {\em Lagrangian intersection {F}loer theory: anomaly and obstruction.
  {P}art {I}}, vol.~46 of {\em AMS/IP Studies in Advanced Mathematics}.
\newblock American Mathematical Society, Providence, RI; International Press,
  Somerville, MA, 2009.

\bibitem{foooII}
{\sc Fukaya, K., Oh, Y.-G., Ohta, H., and Ono, K.}
\newblock {\em Lagrangian intersection {F}loer theory: anomaly and obstruction.
  {P}art {II}}, vol.~46 of {\em AMS/IP Studies in Advanced Mathematics}.
\newblock American Mathematical Society, Providence, RI; International Press,
  Somerville, MA, 2009.

\bibitem{FO3-toricI}
{\sc Fukaya, K., Oh, Y.-G., Ohta, H., and Ono, K.}
\newblock Lagrangian {F}loer theory on compact toric manifolds. {I}.
\newblock {\em Duke Math. J. 151}, 1 (2010), 23--174.

\bibitem{mcduffpol-packing}
{\sc McDuff, D., and Polterovich, L.}
\newblock Symplectic packings and algebraic geometry.
\newblock {\em Invent. Math. 115}, 3 (1994), 405--434.
\newblock With an appendix by Yael Karshon.

\bibitem{ms}
{\sc McDuff, D., and Salamon, D.}
\newblock {\em Introduction to symplectic topology}, second~ed.
\newblock Oxford Mathematical Monographs. The Clarendon Press, Oxford
  University Press, New York, 1998.

\bibitem{msjholo}
{\sc McDuff, D., and Salamon, D.}
\newblock {\em {$J$}-holomorphic curves and symplectic topology}, second~ed.,
  vol.~52 of {\em A. M. S. Colloquium Publications}.
\newblock American Mathematical Society, Providence, RI, 2012.

\bibitem{oh-floercoho}
{\sc Oh, Y.-G.}
\newblock Floer cohomology of {L}agrangian intersections and pseudo-holomorphic
  disks. {I}.
\newblock {\em Comm. Pure Appl. Math. 46}, 7 (1993), 949--993.

\bibitem{oh-adde}
{\sc Oh, Y.-G.}
\newblock Addendum to: ``{F}loer cohomology of {L}agrangian intersections and
  pseudo-holomorphic disks. {I}.'' [{C}omm.\ {P}ure {A}ppl.\ {M}ath.\ {\bf 46}
  (1993), no.\ 7, 949--993; {MR}1223659 (95d:58029a)].
\newblock {\em Comm. Pure Appl. Math. 48}, 11 (1995), 1299--1302.

\bibitem{oh-relative}
{\sc Oh, Y.-G.}
\newblock Relative {F}loer and quantum cohomology and the symplectic topology
  of {L}agrangian submanifolds.
\newblock In {\em Contact and symplectic geometry ({C}ambridge, 1994)}, vol.~8
  of {\em Publ. Newton Inst.} Cambridge Univ. Press, Cambridge, 1996,
  pp.~201--267.

\end{thebibliography}
\end{document}